\documentclass[12pt]{article}

\usepackage{amssymb}
\usepackage{amsthm}
\usepackage{amsmath}
\usepackage{graphicx}
\usepackage{fullpage}
\usepackage{color}
\usepackage{enumerate}
 \numberwithin{equation}{section}
\usepackage{booktabs}
\allowdisplaybreaks

\usepackage[pdftex]{hyperref}
 \usepackage{mathtools}
 \mathtoolsset{showonlyrefs}

\theoremstyle{plain}
\newtheorem{thm}{Theorem}[section]
\newtheorem{cor}[thm]{Corollary}

\newtheorem{prop}[thm]{Proposition}

\theoremstyle{definition}

\theoremstyle{remark}

\newtheorem{rem}[thm]{Remark}

\newcommand{\N}{\mathbb{N}}
\newcommand{\R}{\mathbb{R}}

\newcommand{\bp}{\begin{proof}[\ensuremath{\mathbf{Proof}}]}
\newcommand{\bs}{\begin{proof}[\ensuremath{\mathbf{Solution}}]}
\newcommand{\ep}{\end{proof}}
\newcommand{\be}{\begin{equation}}
\newcommand{\ee}{\end{equation}}

\begin{document}
\title{Extremals in nonlinear potential theory}

\author{Ryan Hynd\footnote{Department of Mathematics, University of Pennsylvania.  Partially supported by NSF grant DMS-1554130.}\;  and Francis Seuffert\footnote{Department of Mathematics, University of Pennsylvania.}}

\maketitle 

\begin{abstract}
We consider the PDE $-\Delta_pu=\rho$, where $\rho$ is a signed Borel measure on $\R^n$. For each $p>n$, we characterize solutions as extremals of a generalized Morrey inequality determined by $\rho$. 
\end{abstract}


\section{Introduction}
In this paper, we will study functional inequalities on the homogeneous Sobolev space 
\be
{\cal  D}^{1,p}(\R^n):=\left\{u\in L^1_{\text{loc}}(\R^n): u_{x_i}\in L^p(\R^n)\;\text{for $i=1,\dots, n$}\right\}
\ee
for 
$$
p>n.
$$
The best known functional inequality on this space is Morrey's inequality which asserts that there is a constant $C$ depending only on $n$ and $p$ such that
\be\label{MorreyIneq}
[u]_{1-n/p}\le C\|Du\|_p
\ee
for all $u\in {\cal  D}^{1,p}(\R^n)$ \cite{MR3409135, MR1814364, MR1501936, MR2492985,MR0109940}.  Here we have written
\be
[u]_{1-n/p}:=\sup_{x\ne y}\left\{\frac{|u(x)-u(y)|}{|x-y|^{1-n/p}}\right\}
\ee
and 
\be
\|Du\|_p:=\left(\int_{\R^n}|Du|^pdx\right)^{1/p}.
\ee
Moreover,  we have identified each $u\in {\cal D}^{1,p}(\R^n)$ with its $1-n/p$ H\"older continuous representative.

\par An extremal of Morrey's inequality is a function $u\in {\cal  D}^{1,p}(\R^n)$ for which equality holds in \eqref{MorreyIneq}. 
In recent work, we established that a nonconstant extremal $u$ of Morrey's inequality exists \cite{HyndSeuf}.  In particular, 
we characterized $u$ as a solution of the PDE 
\be\label{TwoDeltaPDE}
-\Delta_pu=c(\delta_{x_0}-\delta_{y_0})
\ee
in $\R^n$ for some constant $c$.  That is, 
$$
\int_{\R^n}|Du|^{p-2}Du\cdot Dvdx=c(v(x_0)-v(y_0))
$$
for each $v\in {\cal D}^{1,p}(\R^n)$. In equation \eqref{TwoDeltaPDE}, $x_0,y_0\in \R^n$ are distinct points for which 
\be\label{TwoPointForMorrey}
[u]_{1-n/p}=\frac{|u(x_0)-u(y_0)|}{|x_0-y_0|^{1-n/p}}.
\ee
It also follows from the variational structure of this PDE that $u$ is an extremal satisfying \eqref{TwoPointForMorrey} if and only if
$$
\|Du\|_p\le \|Dv\|_p
$$
for each $v\in {\cal D}^{1,p}(\R^n)$ with
$$
v(x_0)-v(y_0)=u(x_0)-u(y_0).
$$

\par It turns out that it is possible to extend these results to a class of inequalities which generalize Morrey's inequality.  To this end, we will 
consider signed Borel measures $\rho$ on $\R^n$ which have the following properties
\be\label{rhoAssump}
\begin{cases}
\text{the support of $\rho$ is compact},\\\\
\rho(\R^n)=0, \text{ and}\\\\
\displaystyle\int_{\R^n}yd\rho(y)\neq 0.
\end{cases}
\ee
For such a measure $\rho$, we define the seminorm
\be\label{rhoSemiNorm}
[u]_\rho:=\sup_{S\in {\cal S}(n)}\left\{\frac{\displaystyle\left|\int_{\R^n}u(x)d\mu(x)\right|}{\displaystyle\left|\int_{\R^n}xd\mu(x)\right|^{1-n/p}}: \mu=S{_\#}\rho\right\},\quad u\in {\cal D}^{1,p}(\R^n).
\ee

\par Here ${\cal S}(n)$ is the collection of similarity transformations of $\R^n$.  That is, each $S\in {\cal S}(n)$ is of the form 
\be\label{SimilarityPrototype}
S(y)=\lambda Oy+z, \quad y\in \R^n
\ee
for some $\lambda>0$, $O\in\textup{O}(n)$, and $z\in \R^n$. Further, $ S{_\#}\rho$ is the usual push forward measure defined as
$$
(S{_\#}\rho)(A):=\rho(S^{-1}(A))
$$
for Borel $A\subset \R^n$. Equivalently 
$$
\int_{\R^n}w(x)d(S{_\#}\rho)(x)=\int_{\R^n}w(S(y))d\rho(y)
$$
for each continuous $w: \R^n\rightarrow \R$.

\par It is not hard to check that if
$$
\rho=\delta_{x_0}-\delta_{y_0}
$$
for any two distinct $x_0, y_0\in \R^n$, then 
$$
[u]_\rho=[u]_{1-n/p}
$$
for each $u\in{\cal D}^{1,p}(\R^n)$.   A more general example of $\rho$ satisfying \eqref{rhoAssump} is 
$$
\rho=\sum^N_{i=1}c_i\delta_{y_i},
$$
where $\sum^N_{i=1}c_i=0$ and $\sum^N_{i=1}c_iy_i\neq 0$.  Another example occurs whenever 
$$
\rho=\text{div} F
$$
for an appropriate class of mappings $F:\R^n\rightarrow \R^n$ with $\int_{\R^n}F(y)dy\neq 0$. We will also explain 
how to associate such an $F$ with each $\rho$ satisfying \eqref{rhoAssump} when $n=1$.

\par We will show that each $\rho$ satisfying \eqref{rhoAssump} leads to a generalized Morrey inequality on ${\cal D}^{1,p}(\R^n)$. In particular, we shall make the basic observation that there is a constant $C$ depending only on $n,p,$ and $\rho$ such that 
\be\label{GeneralizedMorrey}
[u]_\rho\le C\|Du\|_p,\quad u\in {\cal D}^{1,p}(\R^n).
\ee
Furthermore, we will give a simple proof of the existence of a nonconstant extremal $u$ of this inequality and characterize extremals as follows. 
\begin{thm}\label{GenEquivalent}
Suppose $S\in {\cal S}(n)$ and set $\mu=S{_\#}\rho$.     The following are equivalent. 

\begin{enumerate}[(i)]

\item $u\in {\cal D}^{1,p}(\R^n)$ is an extremal of \eqref{GeneralizedMorrey} with 
\be\label{seminormMaxu}
[u]_\rho=\frac{\displaystyle\left|\int_{\R^n}u(x)d\mu(x)\right|}{\displaystyle\left|\int_{\R^n}xd\mu(x)\right|^{1-n/p}}.
\ee

\item There is $c\in \R$ for which
\be\label{MuDeltaPDE}
-\Delta_pu=c\mu
\ee
in $\R^n$.

\item For each $v\in {\cal  D}^{1,p}(\R^n)$ with 
$$
\int_{\R^n}v(x)d\mu(x)=\int_{\R^n}u(x)d\mu(x),
$$
the following inequality holds
$$
\|Du\|_p\le \|Dv\|_p.
$$
\end{enumerate}
\end{thm}

\par This theorem has several corollaries.  Part $(iii)$ implies that extremals of \eqref{GeneralizedMorrey} are conveniently generated by a minimizing the seminorm 
$$
{\cal D}^{1,p}(\R^n)\ni v\mapsto \|Dv\|_p
$$
subject to the constraint 
$$
\int_{\R^n}vd\rho=1.
$$
We will also see how Theorem \ref{GenEquivalent} implies that an extremal $u$ satisfying \eqref{seminormMaxu} is uniformly bounded with
$$
\min_{\textup{supp}(\mu)}u\le u(x)\le \max_{\textup{supp}(\mu)}u, \quad x\in \R^n 
$$
and that the limit
$$
\lim_{|x|\rightarrow\infty}u(x)
$$
exists. Another consequence of this theorem is that extremals are essentially uniquely determined, which will enable us to verify that extremals inherit symmetry and antisymmetry properties from $\rho$.

\par In addition, we will argue that condition $(i)$ in Theorem \ref{GenEquivalent} is always satisfied in the sense that for each $v\in {\cal D}^{1,p}(\R^n)$ there is $S\in {\cal S}(n)$ such that 
\be
[v]_\rho=\frac{\displaystyle\left|\int_{\R^n}v(x)d\mu(x)\right|}{\displaystyle\left|\int_{\R^n}xd\mu(x)\right|^{1-n/p}}
\ee
for $\mu=S{_\#}\rho$. Even more remarkably, we can exploit this fact to refine the generalized Morrey inequality \eqref{GeneralizedMorrey} with the following stability estimates.

\begin{thm}\label{StabilityThm}
Suppose $v\in {\cal D}^{1,p}(\R^n)$ and $C$ is a constant for which \eqref{GeneralizedMorrey} holds. There is an extremal  $u\in {\cal D}^{1,p}(\R^n)$ such that 
\be
\left(\frac{C}{2}\right)^p\|Du-Dv\|_p^p+[v]_\rho^p\le C^p\|Dv\|_p^p
\ee
for $2<p<\infty$ and 
\be
\left(\frac{C}{2}\right)^{\frac{1}{p-1}}\|u'-v'\|_p^{\frac{p}{p-1}}+[v]_\rho^{\frac{p}{p-1}}\le C^{\frac{p}{p-1}}\|v'\|_p^{\frac{p}{p-1}}
\ee 
for $1<p\le 2$. 
\end{thm}

\par The reader will observe that the proofs of our main result are not especially difficult and only require elementary variational arguments. We emphasize that the main point of this note is to highlight the connection between the PDE \eqref{MuDeltaPDE} and the generalized Morrey inequality \eqref{GeneralizedMorrey}. PDEs such as \eqref{MuDeltaPDE} arise in nonlinear potential theory and have been studied in great depth for many years \cite{MR2867756,MR2731714,MR4030249,MR2729305,MR2823872,MR2958962,MR3869444,MR1951245,MR3004772,MR4046205}.  Stability estimates for Sobolev inequalities have also been of great interest in analysis and geometry \cite{MR2538501,MR2395175,MR3896203,MR3404715,MR4048334,1610.06869,MR3695890} and this work provides a simple method to obtain such estimates for Morrey type inequalities.

\par  This paper is organized as follows. In section \ref{PrelimSect}, we show that the seminorm \eqref{rhoSemiNorm} is controlled by a constant times the $1-n/p$ H\"older seminorm and that the maximum ratio determining the seminorm \eqref{rhoSemiNorm} is always achieved. Next we prove nonconstant extremals of \eqref{GeneralizedMorrey} exist in section \ref{NoncontSect}. Then in section \ref{EquivSect}, we prove Theorem \ref{GenEquivalent} and discuss its various corollaries.  In section \ref{StabilitySect}, we issue a short proof of Theorem \ref{StabilityThm}. Finally, in section \ref{ExampleSect}, we write down the extremals and best constant in one spatial dimension and give a duality formula for the sharp constant in any dimension.

\section{Preliminaries}\label{PrelimSect}
We will first verify the basic assertions that $[u]_\rho$ is controlled by the $1-n/p$ H\"older seminorm of $u$ and that $[u]_\rho=0$ implies $u$ is identically equal to a constant. 
\begin{prop}\label{SeminormProp}
$(i)$ There is a constant $A$ such that 
$$
[u]_\rho\le A[u]_{1-n/p}
$$
for all $u\in {\cal D}^{1,p}(\R^n)$. $(ii)$ If 
\be\label{constantCond}
[u]_\rho=0,
\ee
then $u$ is constant throughout $\R^n$.
\end{prop}
\begin{proof}
$(i)$ Suppose $\mu=S{_\#}\rho$, where $S(y)=\lambda Oy+z$ for some $\lambda>0$, $O\in\textup{O}(n)$ and $z\in \R^n$.  Note 
that
\begin{align*}
\int_{\R^n}u(x)d\mu(x)&=\int_{\R^n}(u(x)-u(z))d\mu(x)\\
&=\int_{\R^n}(u(\lambda Oy+z)-u(z))d\rho(y)\\
&\le [u]_{1-n/p}\int_{\R^n}|(\lambda Oy+z)-z|^{1-n/p}d|\rho|(y)\\
&= [u]_{1-n/p}\int_{\R^n}|y|^{1-n/p}d|\rho|(y)\cdot \lambda^{1-n/p}.
\end{align*}
Also observe 
\begin{align*}
\left|\int_{\R^n}xd\mu(x)\right|^{1-n/p}&=\left|\int_{\R^n}(\lambda Oy+z)d\rho(y)\right|^{1-n/p}\\
&=\left|\int_{\R^n}yd\rho(y)\right|^{1-n/p}\cdot \lambda^{1-n/p}.
\end{align*}
Consequently, 
$$
\frac{\displaystyle\left|\int_{\R^n}u(x)d\mu(x)\right|}{\displaystyle\hspace{.2in}\left|\int_{\R^n}xd\mu(x)\right|^{1-n/p}}
\le \frac{\displaystyle\int_{\R^n}|y|^{1-n/p}d|\rho|(y)}{\displaystyle\left|\int_{\R^n}yd\rho(y)\right|^{1-n/p}}\cdot [u]_{1-n/p}=A\cdot [u]_{1-n/p}.
$$

\par $(ii)$  If \eqref{constantCond} holds, then 
$$
\int_{\R^n}\lambda^{n/p-1}u(\lambda Oy+z)d\rho(y)=0
$$
for all $\lambda>0$, $z\in \R^n$ and $O\in \textup{O}(n)$.  In particular,
$$
\int_{\R^n}\frac{u(\lambda Oy+z)-u(z)}{\lambda}d\rho(y)=0
$$
for all $\lambda>0$, $z\in \R^n$ and $O\in \textup{O}(n)$.

\par By Rademacher's Theorem, $u$ is differentiable almost everywhere (Theorem 6.5 in \cite{MR3409135}). Let $z$ be a point in which $Du(z)$ 
exists.  As the support of $\rho$ is compact,  
$$
\lim_{\lambda\rightarrow 0^+}\frac{u(\lambda Oy+z)-u(z)}{\lambda}=Du(z)\cdot Oy
$$
uniformly for each $y\in \text{supp}(\rho)$.  It follows that 
$$
0=\lim_{\lambda\rightarrow 0^+}\int_{\R^n}\frac{u(\lambda Oy+z)-u(z)}{\lambda}d\rho(y)=Du(z)\cdot O\left(\int_{\R^n}yd\rho(y)\right).
$$
As $\int_{\R^n}yd\rho(y)\neq 0$ and $O\in \textup{O}(n)$ is arbitrary, $Du(z)=0$. Since $Du$ vanishes almost everywhere, it must be that $u$ is constant throughout $\R^n$. 
\end{proof}

\par Let us recall the limits 
\be\label{LocalMorreyConsequence}
\displaystyle\lim_{|x-y|\rightarrow 0}\frac{|u(x)-u(y)|}{|x-y|^{1-n/p}}=0
\ee
and
\be\label{MaxHolderRatio}
\displaystyle\lim_{|x|+|y|\rightarrow \infty}\frac{|u(x)-u(y)|}{|x-y|^{1-n/p}}=0,
\ee
which hold for each $u\in {\cal D}^{1,p}(\R^n)$.  These limits were verified in Theorem 6.1 of \cite{HyndSeuf} which 
asserts that for each $u\in {\cal D}^{1,p}(\R^n)$ there are distinct $x_0,y_0$ for which 
$$
[u]_{1-n/p}=\frac{|u(x_0)-u(y_0)|}{|x_0-y_0|^{1-n/p}}.
$$
We will extend this assertion in the following proposition. 
\begin{prop}\label{MaxRhoSemi}
Suppose $u\in{\cal D}^{1,p}(\R^n)$. There is $S\in {\cal S}(n)$ such that 
$$
[u]_\rho=\frac{\displaystyle\left|\int_{\R^n}u(x)d\mu(x)\right|}{\displaystyle\left|\int_{\R^n}xd\mu(x)\right|^{1-n/p}}
$$
for $\mu=S{_\#}\rho$. 
\end{prop}
\begin{proof}
We may that assume that $u$ is nonconstant; or else the assertion holds for any $S\in {\cal S}(n)$. In this case, we can choose sequences $(\lambda_k)_{k\in \N}\subset (0,\infty)$, $(O_k)_{k\in \N}\subset \textup{O}(n)$, $(z_k)_{k\in \N}\subset \R^n$ such that
\be
[u]_\rho=\lim_{k\rightarrow\infty}\frac{\displaystyle\left|\int_{\R^n}\lambda_k^{{n/p}-1}u(\lambda_k O_ky+z_k)d\rho(y)\right|}{\displaystyle\left|\int_{\R^n}yd\rho(y)\right|^{1-n/p}}.
\ee
We may also rewrite this limit as 
\be\label{rewritingLimituKay}
\displaystyle\left|\int_{\R^n}yd\rho(y)\right|^{1-n/p}[u]_\rho=\lim_{k\rightarrow\infty}\left|\int_{\R^n}
\frac{u(\lambda_k O_ky+z_k)-u(z_k)}{|(\lambda_k O_ky+z_k)-z_k|^{1-n/p}}|y|^{1-n/p}d\rho(y)\right|.
\ee
In addition, observe that
\be\label{BasicHolderBoundukay}
\frac{|u(\lambda_k O_ky+z_k)-u(z_k)|}{|(\lambda_k O_ky+z_k)-z_k|^{1-n/p}}\le [u]_{1-n/p}.
\ee
for $k\in \N$ and $y\in \R^n$.

\par Suppose any one of the limits hold
\be
\begin{cases}
\liminf_{k\rightarrow\infty}\lambda_k=0,
\\\\
\limsup_{k\rightarrow\infty}\lambda_k=\infty,
\\\\
\limsup_{k\rightarrow\infty}|z_k|=\infty.
\end{cases}
\ee
In view of \eqref{LocalMorreyConsequence} and \eqref{MaxHolderRatio}, there are subsequences $(\lambda_{k_j})_{k\in \N}\subset (0,\infty)$, $(O_{k_j})_{k\in \N}\subset \textup{O}(n)$, $(z_{k_j})_{k\in \N}\subset \R^n$ such that 
$$
\lim_{j\rightarrow\infty}\frac{|u(\lambda_{k_j} O_{k_j}y+z_{k_j})-u(z_{k_j})|}{|(\lambda_{k_j} O_{k_j}y+z_{k_j})-z_{k_j}|^{1-n/p}}=0
$$
for all $y\in \R^n$. We can then combine this limit with \eqref{BasicHolderBoundukay} and apply dominated convergence to get 
$$
\lim_{k\rightarrow\infty}\int_{\R^n}
\frac{u(\lambda_k O_ky+z_k)-u(z_k)}{|(\lambda_k O_ky+z_k)-z_k|^{1-n/p}}|y|^{1-n/p}d\rho(y)=0.
$$
It would then follow from \eqref{rewritingLimituKay} that $u$ is constant. 

\par As a result, there are subsequences $(\lambda_{k_j})_{k\in \N}\subset (0,\infty)$, $(O_{k_j})_{k\in \N}\subset \textup{O}(n)$, $(z_{k_j})_{k\in \N}\subset \R^n$ such that 
$$
\lambda_{k_j}\rightarrow\lambda, \quad O_{k_j}\rightarrow O,\quad\text{and}\quad z_{k_j}\rightarrow z
$$
for some $\lambda>0$, $O\in \textup{O}(n)$, and $z\in \R^n$. Here we are using that $\textup{O}(n)$ is compact.  Then 
we can send $k=k_j\rightarrow\infty$ in \eqref{rewritingLimituKay} to get 
\begin{align*}
\displaystyle\left|\int_{\R^n}yd\rho(y)\right|^{1-n/p}[u]_\rho&=\left|\int_{\R^n}
\frac{u(\lambda Oy+z)-u(z)}{|(\lambda Oy+z)-z|^{1-n/p}}|y|^{1-n/p}d\rho(y)\right|\\
&=\left|\int_{\R^n}\lambda^{{n/p}-1}u(\lambda Oy+z)d\rho(y)\right|.
\end{align*}
We conclude upon setting $S(y)=\lambda Oy+z$, defining $\mu=S{_\#}\rho$, and rearranging this equality.
\end{proof}

\section{Nonconstant extremals}\label{NoncontSect}
In view of Proposition \ref{SeminormProp} and Morrey's inequality \eqref{MorreyIneq}, there is a constant $C$ such that the generalized Morrey inequality \eqref{GeneralizedMorrey} holds.  It will be convenient for us to denote $C_*$ as the smallest constant such that the generalized Morrey inequality holds and state the inequality as
\be\label{UpsMorrey}
[u]_\rho\le C_*\|Du\|_p
\ee
for $u\in {\cal D}^{1,p}(\R^n)$. We will now verify that a nontrivial extremal exists.  Along the way, we will use the fact that $u\mapsto [u]_\rho$ and $u\mapsto \|Du\|_p$ are each invariant under the transformations
\be
\begin{cases}
u(x)\mapsto -u(x)\\\\
u(x)\mapsto u(x)+c\\\\
u(x)\mapsto  \lambda^{n/p-1}u(\lambda Ox+z)
\end{cases}
\ee
for each $c\in \R,z\in \R^n, O\in \textup{O}(n),$ and $\lambda>0$.

\begin{prop}\label{ExistenceExtProp}
There is a nonconstant $u\in {\cal D}^{1,p}(\R^n)$ with 
$$
[u]_\rho= C_*\|Du\|_p.
$$
\end{prop}
\begin{proof}
Set 
$$
\Lambda:=\inf\left\{\|Du\|_p:[u]_\rho=1 \right\}
$$
and choose a sequence $(u_k)_{k\in \N}\subset  {\cal D}^{1,p}(\R^n)$ with
$$
\Lambda:=\lim_{k\rightarrow\infty}\|Du_k\|_p
$$
and $[u_k]_\rho=1$ for each $k\in \N$. By Proposition \ref{MaxRhoSemi}, we may also select $\lambda^k>0, O^k\in \text{O}(n), z^k\in \R^n$ 
\be
1=[u_k]_{\rho}=\frac{\displaystyle\left|\int_{\R^n}(\lambda^k)^{n/p-1}u_k(\lambda^k O^ky+z^k)d\rho(y)\right|}{\displaystyle\hspace{.2in}\left|\int_{\R^n}yd\rho(y)\right|^{1-n/p}}
\ee 
for each $k\in \N$.  

\par Define
$$
v_k(y):=(\lambda^k)^{n/p-1}\left\{u_k(\lambda^kO^ky+z^k)-u_k(z^k)\right\}, \quad y\in \R^n.
$$
It follows from the definition of $v_k$ and the invariances of the seminorms $u\mapsto [u]_\rho$ and $u\mapsto \|Du\|_p$ that 
\be
\begin{cases}
v_k(0)=0\\\\
[v_k]_\rho=1\\\\
\frac{\displaystyle\left|\int_{\R^n}v_k(y)d\rho(y)\right|}{\displaystyle\hspace{.2in}\left|\int_{\R^n}yd\rho(y)\right|^{1-n/p}}= 1\\\\
\Lambda=\lim_{k\rightarrow\infty}\|Dv_k\|_p.
\end{cases}
\ee
In view of Morrey's inequality, we have that $(v_k)_{k\in\N}$ is equicontinuous. Since $v_k(0)=0$, this sequence is pointwise uniformly bounded on compact subsets of $\R^n$.  By the Arzel\`a-Ascoli Theorem, there is a subsequence  $(v_{k_j})_{j\in\N}$
and $v:\R^n\rightarrow \R$ such that $v_{k_j}\rightarrow v$ locally uniformly on $\R^n$.  

\par It follows that 
$$
v(0)=0,\quad [v]_{\rho}\le 1,\quad \text{and}\quad \frac{\displaystyle\left|\int_{\R^n}v(y)d\rho(y)\right|}{\displaystyle\hspace{.2in}\left|\int_{\R^n}yd\rho(y)\right|^{1-n/p}}= 1.
$$
In particular, 
$$
 [v]_{\rho}=\frac{\displaystyle\left|\int_{\R^n}v(y)d\rho(y)\right|}{\displaystyle\hspace{.2in}\left|\int_{\R^n}yd\rho(y)\right|^{1-n/p}}= 1,
$$
so $v$ is nonconstant.

\par As $\|Dv_{k_j}\|_p$ is bounded, $Dv_{k_j}$ has a weakly convergent subsequence in $L^p(\R^n; \R^n)$; it is routine to check that $Dv_{k_j}\rightharpoonup Dv$ in $L^p(\R^n; \R^n)$. As a result, $v\in {\cal D}^{1,p}(\R^n)$ and 
$$
\Lambda=\lim_{j\rightarrow\infty}\|Dv_{k_j}\|_p\ge \|Dv\|_p\ge \Lambda.
$$ 
Consequently, for a given $u\in {\cal D}^{1,p}(\R^n)$ which is nonconstant 
$$
\Lambda=\|Dv\|_p\le \frac{\|Du\|_p}{ [u]_\rho}.
$$
Thus, 
$$
 [u]_\rho\le \frac{1}{\Lambda}\|Du\|_p
$$
and equality holds for $v$. It follows that $v$ is the desired nonconstant extremal and $C_*=1/\Lambda$. 
\end{proof}
\begin{cor}\label{ExistCor}
Suppose $S\in {\cal S}(n)$ and set $\mu=S{_\#}\rho$. There is an extremal of \eqref{UpsMorrey} with 
$$
[u]_\rho=\frac{\displaystyle\left|\int_{\R^n}u(x)d\mu(x)\right|}{\displaystyle\left|\int_{\R^n}xd\mu(x)\right|^{1-n/p}}.
$$
\end{cor}
\begin{proof}
Assume $S(y)=\lambda Oy+z$, for some $\lambda>0$, $z\in \R^n$ and $O\in \textup{O}(n)$.  Let $v$ be the extremal constructed in the proof of Proposition \ref{ExistenceExtProp} and define  
$$
u(x):=\lambda^{1-n/p}v\left(O^{-1}\frac{x-z}{\lambda}\right),\quad x\in \R^n.
$$
Then 
$$
[u]_\rho=[v]_\rho= \frac{\displaystyle\left|\int_{\R^n}v(y)d\rho(y)\right|}{\displaystyle\hspace{.2in}\left|\int_{\R^n}yd\rho(y)\right|^{1-n/p}}
$$
and $\|Du\|_p=\|Dv\|_p$, so $u$ is an extremal. Moreover, 
$$
\frac{\displaystyle\left|\int_{\R^n}v(y)d\rho(y)\right|}{\displaystyle\hspace{.2in}\left|\int_{\R^n}yd\rho(y)\right|^{1-n/p}}=\frac{\displaystyle\left|\int_{\R^n}\lambda^{n/p-1}u(\lambda Oy+z)d\rho(y)\right|}{\displaystyle\hspace{.2in}\left|\int_{\R^n}yd\rho(y)\right|^{1-n/p}}=\frac{\displaystyle\left|\int_{\R^n}u(x)d\mu(x)\right|}{\displaystyle\left|\int_{\R^n}xd\mu(x)\right|^{1-n/p}}.
$$ 
\end{proof}

\section{Equivalence theorem}\label{EquivSect}
In this section, we will argue that extremals of the generalized Morrey inequality \eqref{UpsMorrey} are uniquely determined up to similarity transformations, are uniformly bounded, are asymptotically flat, and inherit symmetry and antisymmetry properties of $\rho$.  These features are all consequences of Theorem \ref{GenEquivalent}, so we will start by proving this theorem. 

\begin{proof}[Proof of Theorem \ref{GenEquivalent}]
$(i)\Longrightarrow (ii)$ By assumption,
\be\label{FirstPDEid}
\frac{\displaystyle\left|\int_{\R^n}u(x)d\mu(x)\right|^p}{\hspace{.2in}\displaystyle\left|\int_{\R^n}xd\mu(x)\right|^{p-n}}
= C_*^p\int_{\R^n}|Du|^pdx.
\ee
For $t\in \R$ and $v\in {\cal  D}^{1,p}(\R^n)$, we also have 
\be\label{SecondPDEid}
\frac{\displaystyle\left|\int_{\R^n}u(x)d\mu(x)+t\int_{\R^n}v(x)d\mu(x)\right|^p}{\hspace{.2in}\displaystyle\left|\int_{\R^n}xd\mu(x)\right|^{p-n}}\le C_*^p\int_{\R^n}|Du+tDv|^pdx.
\ee
Subtracting \eqref{FirstPDEid} from \eqref{SecondPDEid}, dividing by $t>0$, and sending $t\rightarrow 0$ gives 
$$
\frac{\displaystyle\left|\int_{\R^n}u(x)d\mu(x)\right|^{p-2}\int_{\R^n}u(x)d\mu(x)\int_{\R^n}v(x)d\mu(x)}{\hspace{.2in}\displaystyle\left|\int_{\R^n}xd\mu(x)\right|^{p-n}}\le C_*^p\int_{\R^n}|Du|^{p-2}Du\cdot Dvdx.
$$
\par Replacing $v$ with $-v$ gives 
$$
\int_{\R^n}|Du|^{p-2}Du\cdot Dvdx=c\int_{\R^n}v(x)d\mu(x)
$$
where
$$
c=\frac{\displaystyle\left|\int_{\R^n}u(x)d\mu(x)\right|^{p-2}\int_{\R^n}u(x)d\mu(x)}{\hspace{.1in}C_*^p\displaystyle\left|\int_{\R^n}xd\mu(x)\right|^{p-n}}.
$$
That is, 
$$
-\Delta_pu=c\mu
$$
in $\R^n$. 

\par $(ii)\Longrightarrow (iii)$ Suppose $v\in {\cal  D}^{1,p}(\R^n)$ with $\displaystyle\int_{\R^n}v(x)d\mu(x)=\int_{\R^n}u(x)d\mu(x).$ Then 
\begin{align*}
 \int_{\R^n}|Dv|^pdx&\ge \int_{\R^n}|Du|^pdx+p\int_{\R^n}|Du|^{p-2}Du\cdot (Dv-Du)dx\\
 &=\int_{\R^n}|Du|^pdx+pc\int_{\R^n}(v(x)-u(x))d\mu(x)\\
 &=\int_{\R^n}|Du|^pdx.
\end{align*}

\par $(iii)\Longrightarrow (i)$ Suppose $w\in {\cal D}^{1,p}(\R^n)$ is an extremal of \eqref{UpsMorrey} with 
$$
[w]_\rho=\frac{\displaystyle\left|\int_{\R^n}w(x)d\mu(x)\right|}{\displaystyle\left|\int_{\R^n}xd\mu(x)\right|^{1-n/p}}.
$$
Such a $w$ exists by Corollary \ref{ExistCor}.  Multiplying $w$ by an appropriate scalar, we may assume that  
$$
\int_{\R^n}w(x)d\mu(x)=\int_{\R^n}u(x)d\mu(x).
$$
In particular, 
$$
[w]_\rho\le [u]_\rho.
$$
\par By assumption,
$$
\|Du\|_p\le \|Dw\|_p.
$$
As
$$
[w]_\rho\le [u]_\rho\le C_*\|Du\|_p\le C_*\|Dw\|_p=[w]_\rho,
$$
$u$ is an extremal and 
$$
[u]_\rho= [w]_\rho=\frac{\displaystyle\left|\int_{\R^n}w(x)d\mu(x)\right|}{\displaystyle\left|\int_{\R^n}xd\mu(x)\right|^{1-n/p}}=\frac{\displaystyle\left|\int_{\R^n}u(x)d\mu(x)\right|}{\displaystyle\left|\int_{\R^n}xd\mu(x)\right|^{1-n/p}}.
$$
\end{proof}

\subsection{Uniqueness}
We will now explain that any two extremals are uniquely determined up to a similarity transformation and a few constants. 

\begin{prop}\label{UniqueProp}
Suppose $S\in{\cal S}(n)$ and set $\mu=S{_\#}\rho$.  If $u, v\in {\cal D}^{1,p}(\R^n)$ are extremals of \eqref{UpsMorrey} with 
$$
[u]_\rho=\frac{\displaystyle\int_{\R^n}u(x)d\mu(x)}{\displaystyle\left|\int_{\R^n}xd\mu(x)\right|^{1-n/p}}>0
\quad \text{and}\quad [v]_\rho=\frac{\displaystyle\int_{\R^n}v(y)d\rho(y)}{\displaystyle\left|\int_{\R^n}yd\rho(y)\right|^{1-n/p}}>0,
$$
then 
\be\label{uniqueEqn}
v(y)=\frac{\displaystyle\int_{\R^n}vd\rho}{\displaystyle\int_{\R^n}ud\mu}\left(u(S(y))-u(S(0))\right)+v(0)
\ee
for all $y\in \R^n$.
\end{prop}
\begin{proof}
Suppose $S(y)=\lambda Oy+z$, where $\lambda>0$, $z\in \R^n$ and $O\in \textup{O}(n)$. 
Set $w(y)=\lambda^{n/p-1}u(\lambda Oy+z)$, and observe that $w$ is an extremal of \eqref{UpsMorrey} which satisfies 
$$
[w]_{\rho}=[u]_\rho=\frac{\displaystyle\int_{\R^n}w(y)d\rho(y)}{\displaystyle\left|\int_{\R^n}yd\rho(y)\right|^{1-n/p}}.
$$
By the proof of Theorem \ref{GenEquivalent}, $w$ satisfies the PDE
$$
-\Delta_pw=\frac{\displaystyle\left(\int_{\R^n}wd\rho\right)^{p-1}}{\displaystyle\left|\int_{\R^n}yd\rho(y)\right|^{p-n}}\rho
$$
in $\R^n$.  We also have 
$$
-\Delta_pv=\frac{\displaystyle\left(\int_{\R^n}vd\rho\right)^{p-1}}{\displaystyle\left|\int_{\R^n}yd\rho(y)\right|^{p-n}}\rho
$$
in $\R^n$. 

\par Let us now define 
$$
\tilde v(y)=\frac{v(y)}{\displaystyle\int_{\R^n}vd\rho}\quad \text{and}\quad \tilde w(y)=\frac{w(y)}{\displaystyle\int_{\R^n}wd\rho}
$$
and note that since $\tilde v,\tilde w\in {\cal D}^{1,p}(\R^n)$ satisfy the same PDE:
$$
\int_{\R^n}|D\tilde v|^{p-2}D\tilde v\cdot D\phi dx=\frac{1}{\displaystyle\left|\int_{\R^n}yd\rho(y)\right|^{p-n}}\int_{\R^n}\phi d\rho=\int_{\R^n}|D\tilde w|^{p-2}D\tilde w\cdot D\phi dx
$$
for all $\phi\in {\cal D}^{1,p}(\R^n)$. Choosing $\phi=\tilde v-\tilde w$ gives 
$$
\int_{\R^n}(|D\tilde v|^{p-2}D\tilde v-|D\tilde w|^{p-2}D\tilde w)\cdot (D\tilde v-D\tilde w)dx=0.
$$
As $\R^n\ni z\mapsto |z|^{p-2}z$ is strictly monotone, $\tilde v-\tilde w$ is constant throughout $\R^n$. 

\par It follows that 
$$
\frac{v(y)}{\displaystyle\int_{\R^n}vd\rho}-\frac{w(y)}{\displaystyle\int_{\R^n}wd\rho}=\frac{v(0)}{\displaystyle\int_{\R^n}vd\rho}-\frac{w(0)}{\displaystyle\int_{\R^n}wd\rho}
$$
for all $y\in \R^n$. Moreover,
\begin{align*}
v(y)&=\frac{\displaystyle\int_{\R^n}vd\rho}{\displaystyle\int_{\R^n}wd\rho}(w(y)-w(0))+v(0)\\
&=\frac{\displaystyle\int_{\R^n}vd\rho}{\displaystyle\int_{\R^n}ud\mu}\left(u(\lambda Oy+z)-u(z)\right)+v(0).
\end{align*}
\end{proof}
\begin{cor}\label{UniqueCor}
Suppose $S\in {\cal S}(n)$ and set $\mu=S{_\#}\rho$.  Assume $u_1, u_2\in {\cal D}^{1,p}(\R^n)$ are extremals of \eqref{UpsMorrey} with 
$$
[u_1]_\rho=\frac{\displaystyle\left|\int_{\R^n}u_1(x)d\mu(x)\right|}{\displaystyle\left|\int_{\R^n}xd\mu(x)\right|^{1-n/p}},
\quad
[u_2]_\rho=\frac{\displaystyle\left|\int_{\R^n}u_2(x)d\mu(x)\right|}{\displaystyle\left|\int_{\R^n}xd\mu(x)\right|^{1-n/p}},
$$
and 
\be\label{youoneyoutwoIntegralsEqual}
\int_{\R^n}u_1(x)d\mu(x)=\int_{\R^n}u_2(x)d\mu(x).
\ee
Then $u_1-u_2$ is constant throughout $\R^n$.
\end{cor}
\begin{proof}
Without loss of generality, we may suppose $\displaystyle\int_{\R^n}u_1(x)d\mu(x)>0$. Let $v$ be an extremal of \eqref{UpsMorrey} with 
$$
[v]_\rho=\frac{\displaystyle\int_{\R^n}v(y)d\rho(y)}{\displaystyle\left|\int_{\R^n}yd\rho(y)\right|^{1-n/p}}>0.
$$
In view of \eqref{uniqueEqn}, we have
\begin{align*}
\frac{v(y)}{\displaystyle\int_{\R^n}vd\rho}-\frac{v(0)}{\displaystyle\int_{\R^n}vd\rho}&=\frac{u_1(S(y))-u_1(S(0))}{\displaystyle\int_{\R^n}u_1d\mu}\\
&=\frac{u_2(S(y))-u_2(S(0))}{\displaystyle\int_{\R^n}u_2d\mu}.
\end{align*}
for all $y\in\R^n$. By assumption \eqref{youoneyoutwoIntegralsEqual}, 
$$
u_1(S(y))-u_2(S(y))=u_1(S(0))-u_2(S(0))
$$
for all $y\in\R^n$. It follows that , $u_1(x)-u_2(x)=u_1(S(0))-u_2(S(0))$ for all $x\in \R^n$.
\end{proof}

\subsection{Pointwise bounds and asymptotic flatness}
In the following proposition, we will use the fact that each $u\in{\cal D}^{1,p}(\R^n)$ is 
continuous and that the support of $\rho$ is compact.  Moreover, the support of $S{_\#}\rho$ is compact for any $S\in {\cal S}(n)$. 
\begin{prop}\label{PointwiseBoundProp}
Suppose $S\in {\cal S}(n)$ and set $\mu=S{_\#}\rho$.  Further suppose
$u$ is an extremal of \eqref{UpsMorrey} with 
$$
[u]_\rho=\frac{\displaystyle\left|\int_{\R^n}u(x)d\mu(x)\right|}{\displaystyle\left|\int_{\R^n}xd\mu(x)\right|^{1-n/p}}.
$$
Then 
$$
\inf_{\R^n}u=\min_{\textup{supp}(\mu)}u\quad \text{and}\quad \sup_{\R^n}u=\max_{\textup{supp}(\mu)}u.
$$
\end{prop}
\begin{proof}
We will show $\sup_{\R^n}u=\max_{\textup{supp}(\mu)}u=:M$. To this end, define
$$
w(x):=\min\left\{u(x),M\right\}, \quad x\in \R^n.
$$
It is routine to check that $w\in {\cal D}^{1,p}(\R^n)$, and it is plain to see that $w(x)=u(x)$ for $x\in \textup{supp}(\mu)$.  In particular, 
$$
\int_{\R^n}w(x)d\mu(x)=\int_{\R^n}u(x)d\mu(x).
$$
Note that this implies 
\be\label{uwrhoLessthanEq}
[u]_\rho\le [w]_\rho.
\ee
In view of Theorem \ref{GenEquivalent}, 
$$
\int_{\R^n}|Du|^pdx\le \int_{\R^n}|Dw|^pdx=\int_{u\le M}|Du|^pdx\le \int_{\R^n}|Du|^pdx.
$$
Combining this with \eqref{uwrhoLessthanEq} we conclude that $w$ is an extremal. By Corollary \ref{UniqueCor}, 
$$
u(x)=w(x)=\min\left\{u(x),M\right\}\le M, \quad x\in \R^n.
$$
\end{proof}
For the moment, let $u$ be the extremal discussed in the previous proposition. By Theorem \ref{GenEquivalent}, $u\in {\cal D}^{1,p}(\R^n)$ is $p$-harmonic on $\R^n\setminus \textup{supp}(\mu)$. In particular, as $\textup{supp}(\mu)$ is compact, $u$ is $p$-harmonic on an exterior domain and is uniformly bounded. We can then conclude the following limits by our recent work \cite{HyndSeuf2}.
\begin{cor}\label{LimitCor}
Assume $n\ge 2$, $S\in {\cal S}(n)$ and set $\mu=S{_\#}\rho$.  Further suppose
$u$ is an extremal of \eqref{UpsMorrey} with 
$$
[u]_\rho=\frac{\displaystyle\left|\int_{\R^n}u(x)d\mu(x)\right|}{\displaystyle\left|\int_{\R^n}xd\mu(x)\right|^{1-n/p}}.
$$
The limit
$$
\lim_{|x|\rightarrow\infty}u(x)
$$
exists and 
$$
\lim_{|x|\rightarrow\infty}|x||Du(x)|=0.
$$
\end{cor}

\begin{rem}
When $n=1$, the limits $\lim_{x\rightarrow\infty}u(x)$ and $\lim_{x\rightarrow-\infty}u(x)$ exist. However, these two limits are typically distinct.  See Remark \ref{nequal1Limit} below for more on this point. 
\end{rem}

\subsection{Symmetry and antisymmetry}
We will now consider the scenario in which $\rho$ is invariant under a similarity transformation. We shall see that 
each extremal of \eqref{UpsMorrey} will have a corresponding invariance. In a similar manner, we will establish how 
the antisymmetry of extremals can be inherited from an antisymmetry property of $\rho$.   In the process, we will make use of the limit 
\be\label{SimilarityLimit}
\lim_{|y|\rightarrow\infty}|S(y)|=\infty,
\ee 
which holds for each $S\in {\cal S}(n)$. 

\begin{prop}\label{MainSymmetryProp}
Assume $n\ge 2$ and $v$ is an extremal for \eqref{UpsMorrey} with 
$$
[v]_\rho=\frac{\displaystyle\left|\int_{\R^n}v(y)d\rho(y)\right|}{\displaystyle\left|\int_{\R^n}yd\rho(y)\right|^{1-n/p}}.
$$
If
$$
T{_\#}\rho=\rho
$$
for some $T\in {\cal S}(n)$, then
$$
v=v\circ T.
$$
\end{prop}
\begin{proof}
Suppose $T(x)=\lambda Ox+z$ for some $\lambda>0$, $O\in\text{O}(n)$ and $z\in \R^n$. We first claim that $\lambda$ is necessarily equal to $1$.  To see this, we note
$$
\int_{\R^n}yd\rho(y)=\int_{\R^n}T(y)d\rho(y)=\lambda O\left(\int_{\R^n}yd\rho(y)\right).
$$
Taking the norm of both sides this equation gives
$$
\left|\int_{\R^n}yd\rho(y)\right|=\lambda\left|\int_{\R^n}yd\rho(y)\right|,
$$
which forces $\lambda=1$. In particular, $T(x)= Ox+z$. 

\par Arguing as we did in the proof of Corollary \ref{ExistCor}, we find $v\circ T$ is an extremal with 
$$
[v\circ T]_\rho=\frac{\displaystyle\left|\int_{\R^n}v\circ T(y)d\rho(y)\right|}{\displaystyle\left|\int_{\R^n}yd\rho(y)\right|^{1-n/p}}.
$$
By assumption,
$$
\int_{\R^n}v(y)d\rho(y)=\int_{\R^n}v\circ T(y)d\rho(y).
$$
It then follows that $v-v\circ T$ is constant throughout $\R^n$ by Corollary \ref{UniqueCor}. In view of \eqref{SimilarityLimit} and Corollary \ref{LimitCor}, 
$$
v(x)-v\circ T(x)=\lim_{|y|\rightarrow\infty}(v(y)-v\circ T(y))=\lim_{|y|\rightarrow\infty}(v(y)-v(y))=0.
$$
for each $x\in \R^n$. 
\end{proof} 
\begin{prop}\label{AntiSymmetryProp}
Assume $n\ge 2$ and $v$ is an extremal for \eqref{UpsMorrey} with 
$$
[v]_\rho=\frac{\displaystyle\left|\int_{\R^n}v(y)d\rho(y)\right|}{\displaystyle\left|\int_{\R^n}yd\rho(y)\right|^{1-n/p}}.
$$
If
$$
T{_\#}\rho=-\rho
$$
for some $T\in {\cal S}(n)$, then
\be\label{AntiSymmV}
v(x)+v(T(x))=v(y)+v(T(y))
\ee
for $x,y\in \R^n$.  Moreover, 
$$
\lim_{|x|\rightarrow\infty}v(x)=\frac{v(y)+v(T(y))}{2}.
$$
for each $y\in \R^n$. 
\end{prop}

\begin{proof}
As in the proof of Proposition \ref{MainSymmetryProp}, we find $w=-v\circ T$ is an extremal and 
$$
[w]_\rho=\frac{\displaystyle\left|\int_{\R^n}w(y)d\rho(y)\right|}{\displaystyle\left|\int_{\R^n}yd\rho(y)\right|^{1-n/p}}.
$$
 Further, 
$$
\int_{\R^n}w(y)d\rho(y)=-\int_{\R^n}v\circ T(y)d\rho(y)=\int_{\R^n}v(y)d\rho(y).
$$
By Corollary \ref{UniqueCor}, $v-w$ is constant in $\R^n$. That is,
$$
v(x)+v(T(x))=v(y)+v(T(y)), \quad x,y\in \R^n.
$$
We can also employ Corollary \ref{LimitCor} and \eqref{SimilarityLimit} once again to find 
$$
v(y)+v(T(y))=\lim_{|x|\rightarrow\infty}(v(x)+v(T(x)))=2\lim_{|x|\rightarrow\infty}v(x).
$$
\end{proof}

\par The example where the above propositions are most useful is  
$$
\rho=\delta_{x_0}-\delta_{y_0}
$$
for distinct $x_0$ and $y_0$. As mentioned, the generalized Morrey extremals for this $\rho$ are extremals of Morrey's inequality \eqref{MorreyIneq}.  Let us see 
how Proposition \ref{MainSymmetryProp} can be used to show that Morrey extremals are axially symmetric about the line passing through the points which maximize its $1-n/p$ H\"older seminorm.

\begin{cor}\label{SymmCor}
Suppose $u\in {\cal D}^{1,p}(\R^n)$ a Morrey extremal which satisfies
$$
[u]_{1-n/p}=\frac{|u(x_0)-u(y_0)|}{|x_0-y_0|^{1-n/p}}.
$$
Then 
\be\label{AxiallySYmm}
u(O(x-x_0)+x_0)=u(x), \quad x\in \R^n
\ee
for each $O\in \textup{O}(n)$ such that
$$
O(y_0-x_0)=y_0-x_0.
$$
\end{cor}
\begin{proof}
Set
$$
T(x)=O(x-x_0)+x_0, \quad x\in \R^n
$$
and note
$$
T{_\#}(\delta_{x_0}-\delta_{y_0})=\delta_{x_0}-\delta_{y_0}.
$$
Proposition \ref{MainSymmetryProp} then applies to give \eqref{AxiallySYmm}.
\end{proof}
We can also use Proposition \ref{AntiSymmetryProp} to show that the Morrey extremal $u$ featured in the previous 
corollary is antisymmetric about the hyperplane in $\R^n$ with normal $x_0-y_0$ and which passes through the midpoint of $x_0$ and $y_0$.
\begin{cor}\label{AntiSymmCor}
Suppose $u\in {\cal D}^{1,p}(\R^n)$ is a Morrey extremal which satisfies
$$
[u]_{1-n/p}=\frac{|u(x_0)-u(y_0)|}{|x_0-y_0|^{1-n/p}}.
$$
Then 
\be\label{AntiSymm}
u\left(x -2\frac{\left((x_0-y_0)\cdot (x-\frac{1}{2}(x_0+y_0)\right)}{|x_0-y_0|^2}(x_0-y_0)\right)-\frac{u(x_0)+u(y_0)}{2}
=-\left(u(x)-\frac{u(x_0)+u(y_0)}{2}\right)
\ee
for each $x\in \R^n$ and 
$$
\lim_{|x|\rightarrow \infty}u(x)=\frac{1}{2}(u(x_0)+u(y_0)).
$$
\end{cor}
\begin{proof}
It is easily seen that
$$
T(x)=x -2\frac{\left((x_0-y_0)\cdot (x-\frac{1}{2}(x_0+y_0)\right)}{|x_0-y_0|^2}(x_0-y_0)
$$
is a similarity transformation of $\R^n$. As $T(x_0)=y_0$ and  $T(y_0)=x_0,$ 
$$
T{_\#}(\delta_{x_0}-\delta_{y_0})=-(\delta_{x_0}-\delta_{y_0}).
$$
By Proposition \ref{AntiSymmetryProp},
$$
u(T(x))+u(x)=u(x_0)+u(T(x_0))=u(x_0)+u(y_0)
$$
which is another way of writing \eqref{AntiSymm}.  Proposition \ref{AntiSymmetryProp} also gives 
$$
\lim_{|x|\rightarrow \infty}u(x)=\frac{1}{2}(u(x_0)+u(T(x_0)))=\frac{1}{2}(u(x_0)+u(y_0)).
$$
\end{proof}
\begin{rem}
We verified Corollaries \ref{SymmCor} and \ref{AntiSymmCor} in previous work based on the fact that the extremal in question satisfies the PDE \eqref{TwoDeltaPDE} \cite{HyndSeuf}.
\end{rem}

\section{Stability}\label{StabilitySect}
In our proof of Theorem \ref{StabilityThm} below, we will make use 
of two classical inequalities due to Clarkson \cite{MR1501880}. The first one is 
\be\label{Clarkson1}
\left\|\frac{Dv-Dw}{2}\right\|_p^p+\left\|\frac{Dv+Dw}{2}\right\|_p^p\le \frac{1}{2}\|Dv\|_p^p+  \frac{1}{2}\|Dw\|_p^p
\ee
and it holds $v,w\in {\cal D}^{1,p}(\R^n)$ and $p>2$. The second inequality is 
\be\label{Clarkson2}
\left\|\frac{v'-w'}{2}\right\|_p^{\frac{p}{p-1}}+\left\|\frac{v'+w'}{2}\right\|_p^{\frac{p}{p-1}}\le \left(\frac{1}{2}\|v'\|_p^p+  \frac{1}{2}\|w'\|_p^p\right)^{\frac{1}{p-1}}
\ee
which holds for $1<p\le 2$ and $v,w\in {\cal D}^{1,p}(\R)$.

\begin{proof}[Proof of Theorem \ref{StabilityThm}]
Assume 
$$
[v]_\rho=\frac{\displaystyle\left|\int_{\R^n}v(x)d\mu(x)\right|}{\displaystyle\left|\int_{\R^n}xd\mu(x)\right|^{1-n/p}}
$$
where $\mu=S{_\#}\rho$ for some $S\in {\cal S}(n)$. Let $u\in {\cal D}^{1,p}(\R^n)$ be an extremal such that 
$$
[u]_\rho=\frac{\displaystyle\left|\int_{\R^n}u(x)d\mu(x)\right|}{\displaystyle\left|\int_{\R^n}xd\mu(x)\right|^{1-n/p}}
$$
and 
$$
\int_{\R^n}u(x)d\mu(x)=\int_{\R^n}v(x)d\mu(x).
$$

\par It is routine to check that 
$$
[v]_\rho=\left[\frac{u+v}{2}\right]_\rho.
$$
If $p>2$, we apply \eqref{GeneralizedMorrey}, \eqref{Clarkson1}, and Theorem \ref{GenEquivalent} to get 
\begin{align*}
\left(\frac{C}{2}\right)^p\|Du-Dv\|_p^p+[v]_\rho^p&=C^p\left\|\frac{Du-Dv}{2}\right\|_p^p+\left[\frac{u+v}{2}\right]_\rho^p\\
&\le C^p\left(\left\|\frac{Du-Dv}{2}\right\|_p^p+\left\|\frac{Du+Dv}{2}\right\|_p^p\right)\\
&\le C^p\left(\frac{1}{2}\|Du\|_p^p+  \frac{1}{2}\|Dv\|_p^p\right)\\
&\le C^p\|Dv\|_p^p.
\end{align*}
If $1<p\le 2$, we employ \eqref{GeneralizedMorrey}, \eqref{Clarkson2}, and Theorem \ref{GenEquivalent} to find 
\begin{align*}
\left(\frac{C}{2}\right)^{\frac{p}{p-1}}\|u'-v'\|_p^{\frac{p}{p-1}}+[v]_\rho^{\frac{p}{p-1}}&=C^{\frac{p}{p-1}}\left\|\frac{u'-v'}{2}\right\|_p^{\frac{p}{p-1}}+\left[\frac{u+v}{2}\right]_\rho^{\frac{p}{p-1}}\\
&\le C^{\frac{p}{p-1}}\left(\left\|\frac{u'-v'}{2}\right\|_p^{\frac{p}{p-1}}+\left\|\frac{u'+v'}{2}\right\|_p^{\frac{p}{p-1}}\right)\\
&\le C^{\frac{p}{p-1}}\left(\frac{1}{2}\|u'\|_p^p+  \frac{1}{2}\|v'\|_p^p\right)^{\frac{1}{p-1}}\\
&\le C^{\frac{p}{p-1}}\|v'\|_p^{\frac{p}{p-1}}.
\end{align*}
\end{proof}

\section{Remarks on the best constant}\label{ExampleSect}
In this final section, we will write down an extremal and the best constant $C_*$ for the generalized Morrey inequality \eqref{UpsMorrey} when $n=1$. We will also show how these ideas translate to a duality formula for the best constant for $n\ge 2$.

\subsection{One spatial dimension}
Suppose $n=1$ and $\rho$ satisfies \eqref{rhoAssump}. We define the distribution function of $\rho$ as follows
$$
F(x):=\rho((-\infty,x]), \quad x\in \R.
$$
From this definition and our assumptions on $\rho$, $F$ is bounded, right continuous, and has bounded variation. Moreover, 
there is $a>0$ such that 
\be\label{CompactSupportF}
F(x)=0\;\text{ for all }\; |x|\ge a.
\ee
We will use this $F$ to express the extremals of \eqref{UpsMorrey} and the best constant. In our formulae below, we will write $q$ for the H\"older conjugate to $p$ 
$$
\frac{1}{p}+\frac{1}{q}=1.
$$

\begin{prop}
The function 
\be\label{vFunction1D}
v(x)=-\int^x_{-\infty}|F(y)|^{q-2}F(y)dy,\quad x\in \R
\ee
is an extremal of \eqref{UpsMorrey} with 
\be\label{voneDcond}
[v]_\rho=\frac{\displaystyle\left|\int^\infty_{-\infty}v(y)d\rho(y)\right|}{\displaystyle\left|\int^\infty_{-\infty}yd\rho(y)\right|^{1-1/p}}.
\ee
Moreover, 
$$
C_*=\frac{\displaystyle\left(\int^\infty_{-\infty}|F(y)|^qdy\right)^{1/q}}{\displaystyle\left|\int^\infty_{-\infty}yd\rho(y)\right|^{1/q}}.
$$
\end{prop}
\begin{proof}
Differentiating $v$, we find 
$$
-|v'|^{p-2}v'=F
$$
almost everywhere in $\R.$ For $\phi \in {\cal D}^{1,p}(\R)$, we can multiply the above equation with $\phi'$ and integrate by parts to get 
\begin{align*}
\int^\infty_{-\infty}|v'(x)|^{p-2}v'(x)\phi'(x)dx&=-\int^\infty_{-\infty}F(x)\phi'(x)dx\\
&=-\int^a_{-a}F(x)\phi'(x)dx\\
&=- F(x)\phi(x)\Big|^a_{-a}+\int_{(-a,a]}\phi(x)d\rho(x)\\
&=\int^\infty_{-\infty}\phi(x)d\rho(x).
\end{align*} 
Here $a$ is chosen as in \eqref{CompactSupportF}; and the integration by parts is justified as $\phi$ is absolutely continuous and $F$ has bounded variation (Theorem 3.3 in \cite{Folland}).   We conclude that $v$ satisfies 
\be\label{onedimODE}
-(|v'|^{p-2}v')'=\rho
\ee
in $\R$. 

\par Theorem \ref{GenEquivalent} then implies that $v$ is an extremal which satisfies \eqref{voneDcond}. Since
$$
\int^\infty_{-\infty}v(y)d\rho(y)=\int^\infty_{-\infty}|v'(y)|^pdy=\int^\infty_{-\infty}|F(y)|^qdy,
$$
we also have
$$
C_*=\frac{[v]_\rho}{\|v'\|_p}=
\frac{\displaystyle\left(\int^\infty_{-\infty}|F(y)|^qdy\right)^{1/q}}{\displaystyle\left|\int^\infty_{-\infty}yd\rho(y)\right|^{1/q}}.
$$
\end{proof} 
\begin{rem}\label{nequal1Limit}
Observe that $v$ defined in \eqref{vFunction1D} satisfies 
$$
v(x)=
\begin{cases}
0, \; &x\le -a
\\\\
-\displaystyle\int^\infty_{-\infty}|F(y)|^{q-2}F(y)dy,\; & x\ge a.
\end{cases}
$$
In particular, 
$$
\lim_{x\rightarrow-\infty}v(x)=0\quad \text{and}\quad \lim_{x\rightarrow\infty}v(x)=-\int^\infty_{-\infty}|F(y)|^{q-2}F(y)dy.
$$
\end{rem}

\subsection{Duality}
The observations we made for $n=1$ can be extended as follows.  We will again use $q$ for the H\"older conjugate to $p$ and say that 
$$
\text{div}F=\rho
$$
in $\R^n$ provided  
$$
-\int_{\R^n}Dv\cdot Fdy=\int_{\R^n}vd\rho
$$
for $v\in{\cal D}^{1,p}(\R^n)$.  Here $F\in L^q(\R^n; \R^n)$.
\begin{prop} Suppose $\rho$ satisfies \eqref{rhoAssump}. Then
\be\label{DualityFormula}
\sup_{\|Du\|_p\le 1}\left|\int_{\R^n}u(y)d\rho(y)\right|=\inf\left\{\left(\int_{\R^n}|F(y)|^qdy\right)^{1/q}: \textup{div}F=\rho\;\textup{in}\;\R^n\right\}.
\ee
\end{prop}
\begin{proof}
Suppose $u\in {\cal D}^{1,p}(\R^n)$ satisfies $\|Du\|_p\le 1$ and that $F\in L^q(\R^n; \R^n)$ fulfills $\text{div}F=\rho$ in $\R^n$. By H\"older's inequality, 
\be\label{DualityFormula2}
\left|\int_{\R^n}u(y)d\rho(y)\right|=\left|\int_{\R^n}Du(y)\cdot F(y)dy\right|
\le \|Du\|_p\left(\int_{\R^n}|F(y)|^qdy\right)^{1/q}\le \left(\int_{\R^n}|F(y)|^qdy\right)^{1/q}.
\ee
This show's that the left hand side of \eqref{DualityFormula} is less than or equal to the right hand side. 

\par We claim that equality holds in \eqref{DualityFormula2} (and therefore in \eqref{DualityFormula}) for an extremal of the generalized Morrey inequality $v$ such that
$$
[v]_\rho=\frac{\displaystyle\int_{\R^n}v(y)d\rho(y)}{\displaystyle\left|\int_{\R^n}yd\rho(y)\right|^{1-n/p}}
$$
and $\|Dv\|_p=1$.  To see this, we recall that any such $v$ satisfies 
$$
-\Delta_pv=c\rho
$$
for some $c>0$. In order to conclude, we note the $L^q(\R^n; \R^n)$ mapping 
$$
F=-\frac{1}{c^{\frac{1}{p-1}}}|Dv|^{p-2}Dv
$$
satisfies $\text{div}F=\rho$ and 
$$
\int_{\R^n}vd\rho=-\int_{\R^n}Dv\cdot Fdy=\frac{1}{c^{\frac{1}{p-1}}}=\left(\int_{\R^n}|F(y)|^qdy\right)^{1/q}.
$$
\end{proof}
The duality formula \eqref{DualityFormula} gives us an expression for $C_*$. 
\begin{cor} Suppose $\rho$ satisfies \eqref{rhoAssump}. The best constant in the corresponding generalized Morrey inequality \eqref{UpsMorrey} is given by
$$
C_*=\frac{\inf\left\{\displaystyle\left(\int_{\R^n}|F(y)|^qdy\right)^{1/q}: \textup{div}F=\rho\;\textup{in}\;\R^n\right\}}{\displaystyle\left|\int_{\R^n}yd\rho(y)\right|^{1-n/p}}.
$$
\end{cor}
\begin{proof}
In view of the generalized Morrey inequality \eqref{UpsMorrey},
$$
\frac{\displaystyle\left|\int_{\R^n}v(y)d\rho(y)\right|}{\displaystyle\left|\int_{\R^n}yd\rho(y)\right|^{1-n/p}}\le C_*\|Dv\|_p
$$
for each $v\in {\cal D}^{1,p}(\R^n)$. By Corollary \ref{ExistCor}, equality holds for an appropriately chosen extremal $v$. In particular,  identity \eqref{DualityFormula} gives
\begin{align*}
\left|\int_{\R^n}yd\rho(y)\right|^{1-n/p}C_*&=\sup_{\|Dv\|_p=1}\left|\int_{\R^n}v(y)d\rho(y)\right|\\
&=\inf\left\{\displaystyle\left(\int_{\R^n}|F(y)|^qdy\right)^{1/q}: \textup{div}F=\rho\;\textup{in}\;\R^n\right\}.
\end{align*}
\end{proof}

\bibliography{MultipoleBib}{}

\begin{thebibliography}{10}

\bibitem{MR2867756}
Anders Bj\"{o}rn and Jana Bj\"{o}rn.
\newblock {\em Nonlinear potential theory on metric spaces}, volume~17 of {\em
  EMS Tracts in Mathematics}.
\newblock European Mathematical Society (EMS), Z\"{u}rich, 2011.

\bibitem{MR2731714}
Verena B\"{o}gelein and Jens Habermann.
\newblock Gradient estimates via non standard potentials and continuity.
\newblock {\em Ann. Acad. Sci. Fenn. Math.}, 35(2):641--678, 2010.

\bibitem{MR2538501}
A.~Cianchi, N.~Fusco, F.~Maggi, and A.~Pratelli.
\newblock The sharp {S}obolev inequality in quantitative form.
\newblock {\em J. Eur. Math. Soc. (JEMS)}, 11(5):1105--1139, 2009.

\bibitem{MR2395175}
Andrea Cianchi.
\newblock Sharp {M}orrey-{S}obolev inequalities and the distance from
  extremals.
\newblock {\em Trans. Amer. Math. Soc.}, 360(8):4335--4347, 2008.

\bibitem{MR4030249}
Andrea Cianchi and Vladimir~G. Maz'ya.
\newblock Optimal second-order regularity for the {$p$}-{L}aplace system.
\newblock {\em J. Math. Pures Appl. (9)}, 132:41--78, 2019.

\bibitem{MR1501880}
James~A. Clarkson.
\newblock Uniformly convex spaces.
\newblock {\em Trans. Amer. Math. Soc.}, 40(3):396--414, 1936.

\bibitem{MR2729305}
Frank Duzaar and Giuseppe Mingione.
\newblock Gradient continuity estimates.
\newblock {\em Calc. Var. Partial Differential Equations}, 39(3-4):379--418,
  2010.

\bibitem{MR2823872}
Frank Duzaar and Giuseppe Mingione.
\newblock Gradient estimates via non-linear potentials.
\newblock {\em Amer. J. Math.}, 133(4):1093--1149, 2011.

\bibitem{MR3409135}
Lawrence~C. Evans and Ronald~F. Gariepy.
\newblock {\em Measure theory and fine properties of functions}.
\newblock Textbooks in Mathematics. CRC Press, Boca Raton, FL, revised edition,
  2015.

\bibitem{MR3896203}
Alessio Figalli and Robin Neumayer.
\newblock Gradient stability for the {S}obolev inequality: the case {$p\geq
  2$}.
\newblock {\em J. Eur. Math. Soc. (JEMS)}, 21(2):319--354, 2019.

\bibitem{Folland}
G.~Folland.
\newblock {\em Real analysis}.
\newblock Pure and Applied Mathematics (New York). John Wiley \& Sons, Inc.,
  New York, second edition, 1999.
\newblock Modern techniques and their applications, A Wiley-Interscience
  Publication.

\bibitem{MR3404715}
Nicola Fusco.
\newblock The quantitative isoperimetric inequality and related topics.
\newblock {\em Bull. Math. Sci.}, 5(3):517--607, 2015.

\bibitem{MR1814364}
David Gilbarg and Neil~S. Trudinger.
\newblock {\em Elliptic partial differential equations of second order}.
\newblock Classics in Mathematics. Springer-Verlag, Berlin, 2001.
\newblock Reprint of the 1998 edition.

\bibitem{HyndSeuf}
Ryan Hynd and Francis Seuffert.
\newblock Extremals of {M}orrey's inequality.
\newblock {\em Preprint}, 2018.

\bibitem{HyndSeuf2}
Ryan Hynd and Francis Seuffert.
\newblock Asymptotic flatness of morrey extremals.
\newblock {\em Preprint}, 2019.

\bibitem{MR2958962}
Michael K.-H. Kiessling.
\newblock On the quasi-linear elliptic {PDE}
  {$-\nabla\cdot(\nabla{u}/\sqrt{1-|\nabla{u}|^2})=4\pi\sum_ka_k\delta_{s_k}$}
  in physics and geometry.
\newblock {\em Comm. Math. Phys.}, 314(2):509--523, 2012.

\bibitem{MR3869444}
Michael K.-H. Kiessling.
\newblock Correction to: {O}n the quasi-linear elliptic {PDE}
  {$-\nabla\cdot(\nabla u/\sqrt{1-|\nabla u|^2})=4\pi\sum_ka_k\delta_{s_k}$} in
  physics and geometry [ 2958962].
\newblock {\em Comm. Math. Phys.}, 364(2):825--833, 2018.

\bibitem{MR1951245}
Juha Kinnunen and Olli Martio.
\newblock Nonlinear potential theory on metric spaces.
\newblock {\em Illinois J. Math.}, 46(3):857--883, 2002.

\bibitem{MR3004772}
Tuomo Kuusi and Giuseppe Mingione.
\newblock Linear potentials in nonlinear potential theory.
\newblock {\em Arch. Ration. Mech. Anal.}, 207(1):215--246, 2013.

\bibitem{MR1501936}
Charles~B. Morrey, Jr.
\newblock On the solutions of quasi-linear elliptic partial differential
  equations.
\newblock {\em Trans. Amer. Math. Soc.}, 43(1):126--166, 1938.

\bibitem{MR2492985}
Charles~B. Morrey, Jr.
\newblock {\em Multiple integrals in the calculus of variations}.
\newblock Classics in Mathematics. Springer-Verlag, Berlin, 2008.
\newblock Reprint of the 1966 edition [MR0202511].

\bibitem{MR4048334}
Robin Neumayer.
\newblock A note on strong-form stability for the {S}obolev inequality.
\newblock {\em Calc. Var. Partial Differential Equations}, 59(1):Paper No. 25,
  8, 2020.

\bibitem{MR4046205}
Quoc-Hung Nguyen and Nguyen~Cong Phuc.
\newblock Pointwise gradient estimates for a class of singular quasilinear
  equations with measure data.
\newblock {\em J. Funct. Anal.}, 278(5):108391, 35, 2020.

\bibitem{MR0109940}
L.~Nirenberg.
\newblock On elliptic partial differential equations.
\newblock {\em Ann. Scuola Norm. Sup. Pisa (3)}, 13:115--162, 1959.

\bibitem{1610.06869}
Francis Seuffert.
\newblock A stability result for a family of sharp gagliardo-nirenberg
  inequalities.
\newblock {\em Preprint}, 2016.

\bibitem{MR3695890}
Francis Seuffert.
\newblock An extension of the {B}ianchi-{E}gnell stability estimate to {B}akry,
  {G}entil, and {L}edoux's generalization of the {S}obolev inequality to
  continuous dimensions.
\newblock {\em J. Funct. Anal.}, 273(10):3094--3149, 2017.

\end{thebibliography}

\bibliographystyle{plain}

\typeout{get arXiv to do 4 passes: Label(s) may have changed. Rerun}

\end{document}